\documentclass[11pt,a4paper,twoside,leqno]{amsart}
\usepackage{graphicx}
\usepackage{mathptmx}      
\usepackage{latexsym}
\usepackage[english]{babel}
\usepackage{amsmath,mathrsfs,amssymb}

\usepackage{indentfirst}
\usepackage{hyperref}
\usepackage{braket}
\usepackage{comment}
\usepackage{microtype}
\numberwithin{equation}{section}

\usepackage{tikz}
\usepackage{tikz-cd}
\usepackage{rotating}
\newcommand*{\isoarrow}[1]{\arrow[#1,"\rotatebox{90}{\(\sim\)}"
]}
\usetikzlibrary{matrix,shapes,arrows,decorations.pathmorphing}
\tikzset{commutative diagrams/.cd,
mysymbol/.style={start anchor=center,end anchor=center,draw=none}}

\newcommand{\Z}{\mathbb Z}
\newcommand{\A}{\mathbb A}
\newcommand{\C}{\mathbb C}
\newcommand{\Q}{\mathbb Q}
\newcommand{\DDT}{\mathsf{DT}}
\newcommand{\PPT}{\mathsf{PT}}
\DeclareMathOperator{\Quot}{Quot}
\DeclareMathOperator{\Supp}{Supp}
\DeclareMathOperator{\Hilb}{Hilb}
\DeclareMathOperator{\Chow}{Chow}
\DeclareMathOperator{\Jac}{Jac}
\DeclareMathOperator{\Bl}{Bl}
\DeclareMathOperator{\Spec}{Spec\,}
\DeclareMathOperator{\dd}{d}
\DeclareMathOperator{\Sym}{Sym}
\DeclareMathOperator{\Ext}{Ext}

\newtheoremstyle{thm} 
        {4mm}
        {4mm}
        {\slshape}
        {4mm}
        {\scshape}
        {.}
        {1mm}
        {}
\theoremstyle{thm}
\newtheorem{prop}{Proposition}[section]
\newtheorem*{teo*}{Theorem}
\newtheorem{lemma}[prop]{Lemma}
\newtheorem{corollary}[prop]{Corollary}
\newtheorem{theorem}{Theorem}

\theoremstyle{definition}
\newtheorem{example}[prop]{Example}
\newtheorem{definition}[prop]{Definition}
\newtheorem{remark}[prop]{Remark}

\title{The DT/PT correspondence for smooth curves}
\author{Andrea T.~Ricolfi}
\email{atricolfi@mpim-bonn.mpg.de}
\keywords{Donaldson--Thomas invariants, Hilbert--Chow morphism.}
\subjclass[2010]{Primary 14N35; Secondary 14C05.}

\begin{document}

\begin{abstract}
We show a version of the DT/PT correspondence relating local curve counting invariants, encoding the contribution of a fixed smooth curve in a Calabi--Yau threefold. We exploit a local study of the Hilbert--Chow morphism about the cycle of a smooth curve. We determine, via Quot schemes, the global Donaldson--Thomas theory of a general Abel--Jacobi curve of genus $3$.
\end{abstract}

\maketitle

\tableofcontents

\section{Introduction}
\label{intro}
Let $Y$ be a smooth, projective Calabi--Yau threefold. The Donaldson--Thomas (DT) invariants of $Y$ are enumerative invariants attached to the Hilbert scheme 
\[
I_m(Y,\beta)=\Set{Z\subset Y|\chi(\mathscr{O}_Z)=m,\,[Z]=\beta},
\]
viewed as a moduli space of ideal sheaves \cite{ThomasThesis}. These numbers are insensitive to small deformations of the complex structure of $Y$, but they do change when we perturb the stability condition that defines them: the rules that govern these changes are the so called \emph{wall-crossing formulas}. It is in this spirit that one can interpret the ``DT/PT correspondence'', an equality of generating functions
\begin{equation}\label{generalDTPT}
\DDT_\beta(q)=\DDT_0(q)\cdot \PPT_\beta(q)
\end{equation}
first conjectured in \cite{PT} and later proved by Bridgeland \cite{Bri} and Toda \cite{Toda1}. 
The left hand side of \eqref{generalDTPT} is the Laurent series encoding 
Donaldson--Thomas invariants of $Y$ in the class $\beta\in H_2(Y,\mathbb Z)$, whereas $\PPT_\beta$ encodes the Pandharipande--Thomas (PT) 
invariants of $Y$, defined through the moduli space of stable pairs \cite{PT,BPS},
\[
P_m(Y,\beta)=\Set{(F,s)|\chi(F)=m,\,[\Supp F]=\beta}.
\]
Recall that a pair $(F,s)$, consiting of a one-dimensional coherent sheaf $F$ and a section $s\in H^0(Y,F)$, is said to be \emph{stable} when 
$F$ is pure and the cokernel of $s$ is zero-dimensional. Finally, the formula \cite{BFHilb,JLI,LEPA}
\[
\DDT_0(q)=M(-q)^{\chi(Y)}
\]
determines the zero-dimensional DT theory of $Y$. Here $M(q)=\prod_{k>0}\,(1-q^k)^{-k}$ is the MacMahon function.

\subsection{Main result}
We prove a variant of~\eqref{generalDTPT} in this note. 
Let $C\subset Y$ be a smooth curve of genus $g$ embedded in class $\beta$. For integers $n\geq 0$, we define ``local'' DT invariants
\[
\DDT_{n,C}=\int_{I_n(Y,C)}\nu_I\,\dd\chi
\]
where $\nu_I:I_{1-g+n}(Y,\beta)\rightarrow \Z$ is the Behrend function \cite{Beh} on the Hilbert scheme and $I_n(Y,C)\subset I_{1-g+n}(Y,\beta)$ is the closed subset parametrizing subschemes $Z\subset Y$ containing $C$. For instance, a generic point of $I_n(Y,C)$ represents a subscheme consisting of $C$ along with $n$ distinct points in $Y\setminus C$.  
Similarly, we consider $P_n(Y,C)\subset P_{1-g+n}(Y,C)$, the closed subset parametrizing stable pairs $(F,s)$ such that $\Supp F=C$.
The local invariants on the stable pair side
\[
\PPT_{n,C}=\int_{P_n(Y,C)}\nu_P\,\dd\chi
\]
have been studied in \cite{BPS}. By smoothness of $C$, the local moduli space $P_n(Y,C)$ can be identified with the symmetric product $\Sym^nC$. Let us form the generating functions
\[
\DDT_C(q)=\sum_{n\geq 0}\DDT_{n,C}q^{1-g+n}, \qquad \PPT_C(q)=\sum_{n\geq 0}\PPT_{n,C}q^{1-g+n}.
\]
We say that the DT/PT correspondence holds for $C$ if one has
\begin{equation}\label{eqn:localdtpt}
\DDT_C(q)=\DDT_0(q)\cdot \PPT_C(q).
\end{equation}
We can view \eqref{eqn:localdtpt} as a wall-crossing formula relating the local curve counting invariants attached to $C$. 

Let $n_{g,C}$ be the BPS number of $C\subset Y$ \cite{BPS}. Using the known value of the PT side \cite[Section $3$]{BPS},
\[
\PPT_C(q)=n_{g,C}\cdot q^{1-g}(1+q)^{2g-2},
\]
we proved that the DT/PT correspondence \eqref{eqn:localdtpt} holds for $C$ smooth and \emph{rigid} in \cite[Section $5$]{LocalDT}. The goal of this note is to extend the result to all smooth curves. 

\begin{theorem}\label{thm:dtpt}
Let $Y$ be a smooth, projective Calabi--Yau threefold, $C\subset Y$ a smooth curve. Then the DT/PT correspondence \eqref{eqn:localdtpt} holds for $C$.
\end{theorem}

In fact, the conclusion of the theorem holds for all Cohen--Macaulay curves, by recent work of Oberdieck \cite{Ob1}. While he works with motivic Hall algebras, our method involves a local study of the Hilbert--Chow morphism, and builds upon previous calculations \cite{LocalDT}, especially the weighted Euler characteristic of the Quot scheme $\Quot_n(\mathscr I_C)$, as we explain in Section \ref{section:dtpt}.

\medskip
The local invariants do not depend on the scheme structure one may put on $I_n(Y,C)$ and $P_n(Y,C)$. However, on the DT side, we will exploit the Hilbert--Chow morphism to endow $I_n(Y,C)$ with a natural scheme structure, and we will prove that it agrees with the Quot scheme studied in \cite{LocalDT}.
So, in the local theory, we can think of the moduli spaces
\[
\Quot_n(\mathscr I_C)\qquad\textrm{and}\qquad \Sym^nC
\]
as living on opposite sides of the wall separating DT and PT theory from one another.

\medskip
\noindent
\emph{Conventions.}
All schemes are defined over $\C$. The Calabi--Yau condition for us is simply the existence of a trivialization of the canonical line bundle. The Hilbert--Chow morphism $\Hilb_r(X/S)\rightarrow \Chow_r(X/S)$ is the one constructed by D.~Rydh in \cite{Rydh1}.

\section{The DT/PT correspondence}\label{section:dtpt}
In this section we outline our strategy to deduce Theorem \ref{thm:dtpt}.

\medskip
Let $Y$ be a smooth projective variety, not necessarily Calabi--Yau. We consider the Hilbert--Chow morphism
\begin{equation}\label{hc1638}
\Hilb_1(Y)\rightarrow \Chow_1(Y)
\end{equation}
constructed in \cite{Rydh1}, sending a $1$-dimensional subscheme of $Y$ to its fundamental cycle.  We recall its definition in Section \ref{sec:fib}.
Let $I_{m}(Y,\beta)\subset \Hilb_1(Y)$ be the component parametrizing subschemes $Z\subset Y$ such that 
\[
\chi(\mathscr{O}_Z)=m\in\Z,\qquad [Z]=\beta\in H_2(Y,\Z).
\]
Similarly, we let $\Chow_1(Y,\beta)\subset \Chow_1(Y)$ be the component parametrizing $1$-cycles of degree $\beta$. Then \eqref{hc1638} restricts to a morphism
\[
\mathsf h_m:I_{m}(Y,\beta)\rightarrow \Chow_1(Y,\beta).
\]

\begin{definition}\label{def:fixedcurve}
Fix an integer $n\geq 0$. For a Cohen--Macaulay curve $C\subset Y$ of arithmetic genus $g$ embedded in class $\beta$, we let
\[
I_n(Y,C)\subset I_{1-g+n}(Y,\beta)
\]
denote the scheme-theoretic fibre of $\mathsf h_{1-g+n}$, over the cycle of $C$.
\end{definition}

\begin{remark}\label{rmk:hilbchow}
We will use that \eqref{hc1638} is an isomorphism around normal schemes, at least in characteristic zero \cite[Cor.~$12.9$]{Rydh1}. Thus, for a smooth curve $C\subset Y$, we will identify Chow with Hilb locally around the cycle $[C]\in \Chow_1(Y)$ and the ideal sheaf $\mathscr I_C\in \Hilb_1(Y)$.
For this reason, we do not need the representability of the global Chow functor, as around the point $[C]\in \Chow_1(Y,\beta)$ we can work with the Hilbert scheme $I_{1-g}(Y,\beta)$ instead. 
\end{remark} 

Consider the Quot scheme 
\[
\Quot_n(\mathscr I_C)
\]
parametrizing quotients of length $n$ of the ideal sheaf $\mathscr I_C\subset \mathscr{O}_Y$. 
We proved in \cite[Lemma~$5.1$]{LocalDT} that the association $[\theta:\mathscr I_C\twoheadrightarrow \mathscr E]\mapsto \ker\theta$ defines a closed immersion
\begin{equation}\label{closed918361}
\Quot_n(\mathscr I_C)\hookrightarrow I_{1-g+n}(Y,\beta).
\end{equation}
More precisely, for a scheme $S$, an $S$-valued point of the Quot scheme is a flat quotient $\mathscr E=\mathscr I_{C\times S}/\mathscr I_Z$, and in the short exact sequence
\[
0\rightarrow \mathscr E\rightarrow \mathscr O_Z\rightarrow \mathscr O_{C\times S}\rightarrow 0
\]
over $Y\times S$, the middle term is $S$-flat, so $Z$ defines an $S$-point of $I_{1-g+n}(Y,\beta)$.
The $S$-valued points of the image of \eqref{closed918361} consist precisely of those flat families $Z\subset Y\times S\rightarrow S$ such that $Z$ contains $C\times S$ as a closed subscheme. This will be used implicitly in the proof of Theorem \ref{thm:fibre}.

The schemes $I_n(Y,C)$ and $\Quot_n(\mathscr I_C)$ have the same $\C$-valued points: they both parametrize subschemes $Z\subset Y$ consisting of $C$ together with ``$n$ points'', possibly embedded.
The first step towards Theorem \ref{thm:dtpt} is the following result, whose proof is postponed to the next section.

\begin{theorem}\label{thm:fibre}
Let $Y$ be a smooth projective variety, $C\subset Y$ a smooth curve of genus $g$. Then $I_n(Y,C)=\Quot_n(\mathscr I_C)$ as subschemes of $I_{1-g+n}(Y,\beta)$.
\end{theorem}

As an application of Theorem \ref{thm:fibre}, in Section \ref{sec:ajcurves} we compute the reduced Donaldson--Thomas theory of a general Abel--Jacobi curve of genus $3$.

To proceed towards Theorem \ref{thm:dtpt}, we need to examine the local structure of the Hilbert scheme around subschemes $Z\subset Y$ whose maximal Cohen--Macaulay subscheme
$C\subset Z$ is smooth. The result, given below, will be proven in the next section. 

\begin{theorem}\label{thm:local}
Let $Y$ be a smooth projective variety, $C\subset Y$ a smooth curve of genus $g$. 
Then, locally analytically around $I_n(Y,C)$, the Hilbert scheme $I_{1-g+n}(Y,\beta)$ is isomorphic to $I_n(Y,C)\times \Chow_1(Y,\beta)$.
\end{theorem}

Roughly speaking, this means that the Hilbert--Chow morphism, locally about the cycle
\[
[C]\in \Chow_1(Y,\beta),
\] 
behaves like a fibration with typical fibre $I_n(Y,C)$. 
To obtain this, we first identify $\Chow$ with $\Hilb$ locally around $C$, cf.~Remark \ref{rmk:hilbchow}. We then need to trivialize the universal curve $\mathscr C\rightarrow \Hilb$, which can be done since smooth maps are analytically locally trivial (on the source). However, even if we had $\mathscr C = C\times \Hilb$, we would not be done: the fibre of Hilbert--Chow (which is the Quot scheme by Theorem \ref{thm:fibre}) depends on the embedding of the curve into $Y$, not just on the abstract curve. So to prove Theorem \ref{thm:local} we need to trivialize (locally) the embedding of the universal curve into $Y\times \Hilb$. This is taken care of by a local-analytic version of the tubular neighborhood theorem. After this step, Theorem \ref{thm:local} follows easily. 

\medskip
Granting Theorems \ref{thm:fibre} and \ref{thm:local}, we can prove the DT/PT correspondence for smooth curves. So now we assume $C$ is a smooth curve embedded in class $\beta$ in a smooth, projective Calabi--Yau threefold $Y$.

\begin{proof}[Proof of Theorem \ref{thm:dtpt}]
By \cite[Cor.~$12.9$]{Rydh1}, the Hilbert--Chow morphism 
\[
\mathsf h_{1-g}:I_{1-g}(Y,\beta)\rightarrow \Chow_1(Y,\beta)
\]
is (in characteristic zero) an isomorphism over the locus of normal schemes. Under this local identification, the cycle $[C]$ corresponds to the ideal sheaf $\mathscr I_C$. We let $\nu(\mathscr I_C)$ be the value of the Behrend function on $I_{1-g}(Y,\beta)$ at the point corresponding to $\mathscr I_C$.
Since the Behrend function can be computed locally analytically \cite[Prop.~$4.22$]{Beh}, Theorem \ref{thm:local} implies the identity
\[
\nu_I\big|_{I_n(Y,C)}=\nu(\mathscr I_C)\cdot \nu_{I_n(Y,C)},
\]
where $\nu_I$ is the Behrend function of $I=I_{1-g+n}(Y,\beta)$. After integration, we find
\[
\DDT_{n,C}=\nu(\mathscr I_C)\cdot \tilde\chi(I_n(Y,C)),
\]
where $\tilde\chi(I_n(Y,C))$, by Theorem \ref{thm:fibre}, agrees with the weighted Euler characteristic of the Quot scheme $\Quot_n(\mathscr I_C)$. But we proved in \cite[Thm.~$5.2$]{LocalDT} that the relation
\[
\DDT_{n,C}=\nu(\mathscr I_C)\cdot \tilde\chi(\Quot_n(\mathscr I_C))
\]
is equivalent to the local DT/PT correspondence \eqref{eqn:localdtpt}, so the theorem follows. 
\end{proof}

As observed in \cite{LocalDT}, the local DT/PT correspondence says that the local invariants are determined by the topological Euler characteristic of the corresponding moduli space, along with the BPS number of the fixed smooth curve $C\subset Y$. The latter can be computed as
\[
n_{g,C}=\nu(\mathscr I_C).
\]
For any integer $n\geq 0$, the formulas are
\begin{align*}
\DDT_{n,C}&=n_{g,C}\cdot (-1)^n\chi(I_n(Y,C)),\\
\PPT_{n,C}&=n_{g,C}\cdot (-1)^n\chi(P_n(Y,C)).
\end{align*}
In particular, the local invariants differ by the Euler characteristic of the corresponding moduli space by the \emph{same} constant.

\section{Proofs}\label{sec:proofs}
It remains to prove Theorems \ref{thm:fibre} and \ref{thm:local}.
For Theorem \ref{thm:fibre}, we need to review some definitions and results from \cite{Rydh1}.

\subsection{The fibre of Hilbert--Chow}\label{sec:fib}
Rydh has developed a powerful theory of \emph{relative cycles} and has defined a Hilbert--Chow morphism
\begin{equation}\label{hcmap3746}
\Hilb_r(X/S)\rightarrow \Chow_r(X/S)
\end{equation}
for every algebraic space $X$ locally of finite type over an arbitrary scheme $S$. For us $X$ is always a scheme, projective over $S$. 

We quickly recall the definition of \eqref{hcmap3746}. First of all, the Hilbert scheme $\Hilb_r(X/S)$ parametrizes $S$-subschemes of $X$ that are proper and of dimension $r$ over $S$, 
but not necessarily equidimensional, while $\Chow_r(X/S)$ parametrizes \emph{equidimensional}, proper relative cycles of dimension $r$.
We refer to \cite[Def.~$4.2$]{Rydh1} for the definition of \emph{relative cycles} on $X/S$. Cycles have a (not necessarily equidimensional) support, which 
is a locally closed subset $Z\subset X$. Rydh shows \cite[Prop.~$4.5$]{Rydh1} that if $\alpha$ is a relative cycle on $f:X\rightarrow S$ with support $Z$, 
then, for every $r\geq 0$, on the same family there is a unique \emph{equidimensional} relative cycle $\alpha_r$ with support 
\[
Z_r=\Set{x\in Z|\dim_xZ_{f(x)}=r}\subset Z.
\]
Cycles are called equidimensional when their support is equidimensional over the base.
The essential tool for the definition of~\eqref{hcmap3746} is the \emph{norm family}, defined by the following result.

\begin{theorem}[{\textrm{\cite[Thm.~$7.14$]{Rydh1}}}]
Let $X\rightarrow S$ be a locally finitely presented morphism, $\mathcal F$ a finitely presented $\mathscr{O}_X$-module which is flat over $S$. 
Then there is a canonical relative cycle $\mathcal N_{\mathcal F}$ on $X/S$, with support equal to $\Supp \mathcal F$. 
This construction commutes with arbitraty base change. When $Z\subset X$ is a subscheme which is flat and of finite presentation 
over $S$, we write $\mathcal N_Z=\mathcal N_{\mathscr{O}_Z}$.
\end{theorem}

The Hilbert--Chow functor \eqref{hcmap3746} is defined by $Z\mapsto (\mathcal N_Z)_r$.

\medskip
Even though we do not recall here the full definition of relative cycle, the main idea is the following. For a locally closed subset $Z\subset X$, Rydh defines a \emph{projection of $X/S$ adapted to} $Z$ to be a commutative diagram
\begin{equation}\label{proj1}
\begin{tikzcd}
U\arrow{r}{p}\arrow{d} & X\arrow{dd}\\
B\arrow{d} & \\
T\arrow{r}{g} & S
\end{tikzcd}
\end{equation}
where $U\rightarrow X\times_ST$ is \'etale, $B\rightarrow T$ is smooth and $p^{-1}(Z)\rightarrow B$ is finite. A relative cycle $\alpha$ on $X/S$ with support $Z\subset X$ is the datum, for every projection adapted to $Z$, of a proper family of \emph{zero}-cycles on $U/B$, which Rydh defines as a morphism
\[
\alpha_{U/B/T}:B\rightarrow \Gamma^\star(U/B)
\]
to the scheme of divided powers. We refer to \cite[Def.~$4.2$]{Rydh1} for the additional compatibility conditions that these data should satisfy.

Let now $\mathcal F$ be a flat family of coherent sheaves on $X/S$. If $\mathsf p=(U,B,T,p,g)$ denotes a projection of $X/S$ adapted to $\Supp \mathcal F\subset X$ as in \eqref{proj1}, then the \emph{zero}-cycle defining the norm family $\mathcal N_{\mathcal F}$ at $\mathsf p$ is
\[
(\mathcal N_{\mathcal F})_{U/B/T}=\mathcal N_{p^\ast \mathcal F/B},
\]
constructed in \cite[Cor.~$7.9$]{Rydh1}. For us $\mathcal F$ will always be a structure sheaf, so it will be easy to compare these zero-cycles.

\medskip
If $Z\subset X$ is a subscheme that is smooth over $S$, then the norm family $\mathcal N_Z$ is an example of a \emph{smooth} relative cycle, cf.~\cite[Def.~$8.11$]{Rydh1}.
The next result states an equivalence, in characteristic zero, between smooth relative cycles and subschemes smooth over the base.

\begin{theorem}[{\cite[Thm.~$9.8$]{Rydh1}}]\label{thm:fromrydh}
If $S$ is of characteristic zero, then for every smooth relative cycle $\alpha$ on $X/S$ there is a unique subscheme $Z\subset X$, smooth over $S$, such that $\alpha=\mathcal N_Z$.
\end{theorem}

We can now prove Theorem \ref{thm:fibre}. We fix $Y$ to be a smooth projective variety, $C\subset Y$ a smooth curve of genus $g$ in class $\beta$, 
and we denote by $I_n(Y,C)$ the fibre over $[C]$ of the Hilbert--Chow morphism
\[
I_{1-g+n}(Y,\beta)\rightarrow \Chow_1(Y,\beta),
\]
as in Definition \ref{def:fixedcurve}.

\begin{proof}[Proof of Theorem \ref{thm:fibre}]
We need to show the equality
\[
I_n(Y,C)=\Quot_n(\mathscr I_C)
\]
as subschemes of $I_{1-g+n}(Y,\beta)$. Let $S$ be a scheme over $\C$, and set $X=Y\times_\C S$. Then a family 
\[
Z\subset X\rightarrow S
\]
in the Hilbert scheme is an $S$-valued point of $I_n(Y,C)$ when $(\mathcal N_Z)_1=\mathcal N_{C\times S}$.
The closed immersion \eqref{closed918361} from the Quot scheme to the Hilbert scheme factors through $I_n(Y,C)$. Indeed, any $S$-point $\mathscr I_{C\times S}\twoheadrightarrow \mathscr I_{C\times S}/\mathscr I_Z$ of the Quot scheme gives a closed immersion $C\times S\hookrightarrow Z$ whose relative ideal is zero-dimensional over $S$, thus we have $(\mathcal N_Z)_1=(\mathcal N_{C\times S})_1=\mathcal N_{C\times S}$, where in the second equality we used that $\mathcal N_{C\times S}$ is equidimensional of dimension one over $S$. So we obtain a closed immersion 
\[
\iota:\Quot_n(\mathscr I_C)\hookrightarrow I_n(Y,C).
\]
For every scheme $S$, we have an injective map of sets
\[
\iota(S):\Quot_n(\mathscr I_C)(S)\hookrightarrow I_n(Y,C)(S),
\]
and since $\iota(\Spec \C)$ is a bijection, so far $\iota$ is just a bijective closed immersion. We need to show $\iota(S)$ is onto, and for the moment we deal with the case where $S$ is a fat point. In other words, assume $S$ is the spectrum of a local artinian $\C$-algebra with residue field $\C$. Let $Z\subset X\rightarrow S$ be an $S$-valued point of $I_n(Y,C)$. Consider the finite subscheme $F\subset Y\subset X$ given by the support of $\mathscr I_C/\mathscr I_{Z_0}$, where $Z_0$ is the closed fibre of $Z\rightarrow S$. Form the open set $V=X\setminus F\subset X$. Then we have, as relative cycles on $V/S$, 
\[
(\mathcal N_Z)_1\big|_V=\mathcal N_{C\times S}\big|_V=\mathcal N_{(C\times S)\cap V}.
\]
We claim the left hand side equals the relative cycle $\mathcal N_{Z\cap V}$.
For sure, these two cycles have the same support, as $Z\cap V=Z_1\cap V$, and they are determined by the same set of projections; indeed, being equidimensional of dimension one, they are determined by (compatible data of) relative zero-cycles for every projection $\mathsf p_{V/S}=(U,B,T,p,g)$ such that $B/T$ is smooth of relative dimension one. Let us focus on $(\mathcal N_Z)_1$ first. Here $r=1$ is the maximal relative dimension of a point in $Z$, so the zero-cycle corresponding to a projection $\mathsf p_{X/S}$ as in \eqref{proj1}, and adapted to $Z_1$, is the same as the one defined by the norm family of $Z$ (cf.~the proof of \cite[Prop.~$4.5$]{Rydh1}), namely $\mathcal N_{p^\ast \mathscr{O}_Z/B}$. Now we restrict to the open subset $i:V\rightarrow X$. By definition of pullback, the zero-cycle attached to a projection $\mathsf p_{V/S}$ (adapted to $Z_1\cap V$) is the cycle corresponding to the projection $(U,B,T,i\circ p,g)$ for the full family $Z/S$, namely
\[
\mathcal N_{(i\circ p)^\ast\mathscr{O}_Z/B}=\mathcal N_{p^\ast\mathscr{O}_{Z\cap V}/B}.
\]
The latter is precisely the zero-cycle defining the norm family of $Z\cap V/S$ at the same projection $\mathsf p_{V/S}$, so the claim is proved, 
\[
\mathcal N_{Z\cap V}=(\mathcal N_Z)_1\big|_V.
\]
By the equivalence between smooth cycles and smooth subschemes stated in Theorem \ref{thm:fromrydh}, we conclude that $Z\cap V$ and $(C\times S)\cap V$ are the same (smooth) family over $S$. Moreover, the closure
\[
\overline{(C\times S)\cap V}\subset Z
\]
equals $C\times S$, because the open subscheme $(C\times S)\cap V\subset C\times S$ is fibrewise dense (intersecting with $V$ is only deleting a finite number of points in the special fibre). We have thus reconstructed a closed immersion $C\times S\hookrightarrow Z$, giving a well-defined $S$-valued point of $\Quot_n(\mathscr I_C)$. So $\iota(S)$ is onto, and thus a bijection, whenever $S$ is a fat point. This implies $\iota$ is \'etale, by a simple application of the formal criterion for \'etale maps. The theorem follows because we already know $\iota$ is a bijective closed immersion. 
\end{proof}

\subsection{Local triviality of Hilbert--Chow}

In this section we prove Theorem \ref{thm:local}. The main tool used in the proof is the following local analytic version of the tubular neighborhood theorem.

\begin{lemma}\label{lemma:loctriv}
Let $S$ be a scheme, $j:X\rightarrow Y$ a closed immersion over $S$. Assume $X$ and $Y$ are both smooth over $S$, of relative dimension $d$ and $n$ respectively. Then $j$ is locally analytically isomorphic to the standard linear embedding $\C^d\times S\rightarrow \C^n\times S$. 
\end{lemma}

\begin{proof}
Let $x\in X$ and $y=j(x)\in Y$. Let $\mathscr I\subset \mathscr{O}_Y$ be the ideal sheaf of $X$ in $Y$.
The relative smoothness of $X$, given that of $Y$, is characterized by the Jacobian criterion \cite[Section $8.5$]{Bosch1}, asserting that the short exact sequence
\[
0\rightarrow \mathscr I/\mathscr I^2\rightarrow j^\ast\Omega_{Y/S}\rightarrow \Omega_{X/S}\rightarrow 0
\]
is split locally around $x\in X$. According to \emph{loc.~cit.}~this is also equivalent to the following: whenever we choose local sections $t_1,\dots,t_n$ and $g_1,\dots,g_N$ of $\mathscr{O}_{Y,y}$ such that $\dd t_1,\dots,\dd t_n$ constitute a free generating system for $\Omega_{Y/S,y}$ and $g_1,\dots,g_N$ generate $\mathscr I_y$, after a suitable relabeling we may assume $g_{d+1},\dots,g_n$ generate $\mathscr I$ about $y$ and 
\[
\dd t_1,\dots,\dd t_d,\dd g_{d+1},\dots,\dd g_n
\]
generate $\Omega_{Y/S}$ locally around $y$. In particular, $f_i=t_i\circ j$, for $i=1,\dots,d$, define a local system of parameters at $x$. By this choice of local basis for $\Omega_{Y/S}$ around $y$, we can find open neighborhoods $x\in U\subset X$ and $y\in V\subset Y$ fitting in a commutative diagram
\[
\begin{tikzcd}
U\arrow[hook]{r}{j}\arrow[swap]{d}{\textrm{\'et}} & 
V\arrow{d}{\textrm{\'et}} \\
\A^d_S\arrow[hook]{r} & \A^n_S
\end{tikzcd}
\]
where the vertical maps are defined by the local systems of parameters $(f_1,\dots,f_d)$ and $(t_1,\dots,t_d,g_{d+1},\dots,g_n)$ respectively, and the lower immersion is defined by sending $t_i\mapsto f_i$ for $i=1,\dots,d$ and $g_k\mapsto 0$.
Using the analytic topology, the inverse function theorem allows us to translate the \'etale maps into local analytic isomorphisms, and the statement follows.
\end{proof}

Note that Lemma \ref{lemma:loctriv} does not hold globally. For a closed immersion $X\subset Y$ of smooth complex \emph{projective} varieties, it is not true in general that one can find a global tubular neighborhood. The obstruction lies in $\Ext^1(N_{X/Y},T_X)$.

\medskip
Before the proof of Theorem \ref{thm:local}, we introduce the following notation.
If $Z\subset Y$ is a $1$-dimensional subscheme corresponding to a point in the fibre $I_n(Y,C)$ 
of Hilbert--Chow, we can attach to $Z$ its ``finite part'', the finite subset $F_Z\subset Z$ which is the support of the maximal zero-dimensional subsheaf of $\mathscr{O}_Z$, namely the quotient $\mathscr I_C/\mathscr I_Z$. 

\begin{proof}[Proof of Theorem \ref{thm:local}]
By \cite[Cor.~$12.9$]{Rydh1} the Hilbert--Chow map is a local isomorphism around normal schemes, so we may identify an open neighborhood of the cycle of $C$ in the Chow scheme with an open neighborhood $U$ of $[C]$ in the Hilbert scheme $I_{1-g}(Y,\beta)$. We then consider the Hilbert--Chow map 
\[
\mathsf h=\mathsf h_{1-g+n}:I_{1-g+n}(Y,\beta)\rightarrow \Chow_1(Y,\beta)
\]
and we fix a point in the fibre $[Z_0]\in I_n(Y,C)$. It is easy to reduce to the case where the finite part $F_0=F_{Z_0}\subset Z_0$ is confined on $C$, that is, $Z_0$ has only embedded points. 
We need to show that the Hilbert scheme is locally analytically isomorphic to $U\times I_n(Y,C)$ about $[Z_0]$. By Lemma \ref{lemma:loctriv}, the universal embedding $\mathscr C\subset Y\times U$, locally around the finite set of points $F_0\subset C\subset\mathscr C$, is locally analytically isomorphic to the embedding of the zero section $C\times U\subset C\times U\times \C^2$ of the trivial rank $2$ bundle. 
In particular we can find, in $C\times U\times \C^2$ and in $Y\times U$, analytic open neighborhoods $V$ and $V'$ of $F_0$, fitting in a commutative diagram
\[
\begin{tikzcd}
(C\times U)\cap V\isoarrow{d}\arrow[hook]{r} & V\arrow[hook]{r}{\textrm{open}}\isoarrow{d} & C\times U\times \C^2\\
\mathscr C\cap V'\arrow[hook]{r} & V'\arrow[hook]{r}{\textrm{open}} & Y\times U
\end{tikzcd}
\]
where the vertical maps are analytic isomorphisms. Now consider the open subset 
\[
A=\big\{(Z,u)\in I_n(Y,C)\times U\,\big|\,F_Z\subset V_u\big\}\subset I_n(Y,C)\times U.
\]
Letting $\varphi$ denote the isomorphism $V\,\widetilde{\rightarrow}\,V'$, given a pair $(Z,u)\in A$ we can look at $Z'=\mathscr C_u\cup \varphi(F_Z)$, which is a new subscheme of $Y$, mapping to $u$ under Hilbert--Chow.
The association $(Z,u)\mapsto Z'$ defines an isomorphism between $A$ and the open subset $B\subset \mathsf h^{-1}(U)$ parametrizing subschemes $Z'\subset Y$ such that $F_{Z'}$ is contained in $V'_u$, where $u$ is the image of $[Z']$ under Hilbert--Chow. Note that $[Z_0]\in B$ corresponds to $(Z_0,C)\in A$ under this isomorphism. The theorem is proved.
\end{proof}

\section{The DT theory of an Abel--Jacobi curve}\label{sec:ajcurves}
In this section we fix a non-hyperelliptic curve $C$ of genus $3$, embedded in its Jacobian 
\[
Y=(\Jac C,\Theta)
\]
via an Abel--Jacobi map. We let $\beta=[C]\in H_2(Y,\Z)$ be the corresponding curve class. For $n\geq 0$, we let 
\[
\mathcal H^n_C\subset I_{n-2}(Y,\beta)
\]
be the component of the Hilbert scheme parametrizing subschemes $Z\subset Y$ whose fundamental cycle is algebraically equivalent to $[C]$.

Let $-1:Y\rightarrow Y$ be the automorphism $y\mapsto -y$, and let $-C$ denote the image of $C$. As $C$ is non-hyperelliptic, the cycle of $C$ is not algebraically equivalent to the cycle of $-C$ \cite{Ceresa1}. 
The Hilbert scheme  $I_{n-2}(Y,\beta)$ consists of two connected components, which are interchanged by $-1$. Moreover, the Abel--Jacobi embedding $C\subset Y$ has 
unobstructed deformations, and there is an isomorphism $Y\,\widetilde{\rightarrow}\,\mathcal H^0_C$ given by translations \cite{LangeSernesi}.

\begin{example}
As remarked in \cite[Example $2.3$]{Gulb3}, the morphism
\[
\mathcal H^1_C\rightarrow \mathcal H^0_C\times Y
\]
sending $T_x(C)\cup y\mapsto (T_x(C),y)$, where $T_x$ denotes translation by $x$, is the Albanese map.
It can be easily checked that $\mathcal H^1_C$ is isomorphic to the blow-up
\[
\Bl_{\mathcal U}(\mathcal H^0_C\times Y),
\]
where $\mathcal U$ is the universal family. In particular, $\mathcal H^1_C$ is smooth of dimension $6$.
\end{example}

The quotient of the Hilbert scheme by the translation action of $Y$ gives a Deligne--Mumford stack $I_{m}(Y,\beta)/Y$. 
In fact, since the $Y$-action is free, this is an algebraic space. The \emph{reduced} Donaldson--Thomas invariants
\[
\DDT^Y_{m,\beta}=\int_{I_{m}(Y,\beta)/Y}\nu\,\dd\chi\in\Q
\]
were introduced in \cite{BOPY16} for arbitrary abelian threefolds. We consider their generating function
\[
\DDT_\beta(p)=\sum_{m\in\Z}\DDT^Y_{m,\beta}p^{m}.
\]
We state the following result as a corollary of Theorem \ref{thm:fibre}.

\begin{corollary}\label{cor:dtabeljacobi}
Let $C\subset Y$ be non-hyperelliptic, embedded in class $\beta$. Then
\[
\DDT_\beta(p)=2p^{-2}(1+p)^4.
\]
\end{corollary}

\begin{proof}
As the Hilbert--Chow morphism is an isomorphism around normal schemes, we have an isomorphism
\[
I_{-2}(Y,\beta)\,\widetilde{\rightarrow}\,\Chow_1(Y,\beta).
\]
On the other hand, the Hilbert scheme is the disjoint union of two copies of $\mathcal H^0_C$, where $\mathcal H^0_C\cong Y$ because $C$ is not hyperelliptic. Focusing on the component parametrizing translates of $C$, the Hilbert--Chow morphism $\mathcal H^n_C\rightarrow \mathcal H^0_C$ induces an isomorphism 
\[
Y\times \Quot_n(\mathscr I_C)\,\widetilde{\rightarrow}\, \mathcal H^n_C
\]
by Theorem \ref{thm:fibre}. This shows that the quotient space $\mathcal H^n_C/Y$ is isomorphic to the Quot 
\emph{scheme} $\Quot_n(\mathscr I_C)$. Keeping into account the second component 
of $I_{n-2}(Y,\beta)$, still isomorphic to $\mathcal H^n_C$, we find 
\[
\DDT^Y_{n-2,\beta}=2\cdot \tilde\chi(\Quot_n(\mathscr I_C)),
\]
where $\tilde\chi$ denotes the Behrend weighted Euler characteristic. 
Then 
\[
\DDT_\beta(p)=\sum_{n\geq 0}\DDT^Y_{n-2,\beta}p^{n-2}=2p^{-2}\sum_{n\geq 0}\tilde\chi(\Quot_n(\mathscr I_C))p^n=2p^{-2}(1+p)^4,
\]
where the last equality follows from \cite[Prop.~$5.1$]{LocalDT}.
\end{proof}

If one considers homology classes of type $(1,1,d)$ for all $d\geq 0$, on an arbitrary abelian threefold $Y$, one has the formula
\begin{equation}\label{conj:dtabel}
\sum_{d\geq 0}\sum_{m\in\mathbb Z}\DDT^Y_{m,(1,1,d)}(-p)^mq^d=-K(p,q)^2,
\end{equation}
where $K$ is the Jacobi theta function
\[
K(p,q)=(p^{1/2}-p^{-1/2})\prod_{m\geq 1}\frac{(1-pq^m)(1-p^{-1}q^m)}{(1-q^m)^2}.
\]
Relation \eqref{conj:dtabel} was conjectured in \cite{BOPY16} and proved in \cite{Ob2,Ob3}.
Corollary \ref{cor:dtabeljacobi} confirms the coefficient of $q$ via Quot schemes, when $Y$ is the Jacobian of a general curve. Indeed, in this case the Abel--Jacobi class is of type $(1,1,1)$.

\medskip
The \emph{local} DT theory of a general Abel--Jacobi curve $C$ of genus $3$ is determined as follows. 
Using again the isomorphism $Y\cong\mathcal H^0_C$, we can compute the BPS number
\[
n_{3,C}=\nu(\mathscr I_C)=-1,
\]
thus the DT/PT correspondence at $C$ (Theorem \ref{thm:dtpt}) yields
\[
\DDT_C(q)=\PPT_C(q)=-q^{-2}(1+q)^4.
\]
In other words, the global theory is related to the local one by
\[
\DDT_\beta(q)=-2\cdot \DDT_C(q).
\]

\medskip
\textbf{Acknowledgements}.
I would like to thank my advisors Martin G.~Gulbrandsen and Lars H.~Halle, and Letterio Gatto, for many valuable discussions. Thanks to David Rydh for sharing many useful comments and a key example that was useful to understand the norm family. Finally, I wish to thank Richard Thomas for patiently sharing his insights and commenting on a previous draft of this work.

\bibliographystyle{spmpsci}      
\bibliography{bib}   

\end{document}